\documentclass[11pt]{article}

\usepackage{comment}
\usepackage{amsmath,amssymb,amsthm,comment}
\usepackage{pstricks,pstricks-add,pst-eucl,pst-func,pst-plot}
\usepackage[colorlinks,citecolor=red,urlcolor=blue,linkcolor=blue]{hyperref}
\usepackage{mathrsfs,dsfont}
\usepackage{graphicx}

\usepackage{tikz}

\DeclareRobustCommand*{\sagelogo}{%
  \begin{tikzpicture}[line width=.2ex,line cap=round,rounded corners=.01ex,baseline=-.2ex]
    \draw(0,0) -- (.75em,0) 
      -- (.75em,.7ex) -- (.25em,.7ex) 
      -- (.25em,.75\ht\strutbox) -- (.75em,.75\ht\strutbox);
    \draw(2em,0) -- (1.6em,0)
      -- (1.3em,.75\ht\strutbox) -- (.9em,.75\ht\strutbox)
      -- (.9em,0) -- (1.3em,0)
      -- (1.6em,.75\ht\strutbox) -- (2.1em,.75\ht\strutbox)
      -- (2.1em,-\dp\strutbox) -- (1.45em,-\dp\strutbox);
    \draw(3em,0) -- (2.25em,0)
      -- (2.25em,.75\ht\strutbox) -- (2.8em,.75\ht\strutbox)
      -- (2.8em,.7ex) -- (2.35em, .7ex);
  \end{tikzpicture}%
}

\newtheorem{nummer}{ }
\newtheorem{thm}[nummer]{\sc Theorem}

\newtheorem{cor}[nummer]{\sc Corollary}
\newtheorem{fct}[nummer]{\sc Fact}

\newcommand{\ie} {{\sl i.e.}}
\newcommand{\eg} {{\sl e.g.}}

\newcommand{\dritt}{\#}

\newcommand{\Q}{\mathds{Q}}

\newcommand{\Z}{\mathds{Z}}
\newcommand{\R}{\mathds{R}}

\newcommand{\THM}{\sc Theorem}

\newcommand{\nO}{\mathscr{O}}

\newcommand{\tripleAbf}{\bf rational Pythagorean $\boldsymbol{A}$-triple}
\newcommand{\tripleA}{rational Py\-tha\-go\-re\-an $A$-triple}

\newcommand{\TripleAbf}{\bf Pythagorean $\boldsymbol{A}$-triple}
\newcommand{\TripleA}{Py\-tha\-go\-re\-an $A$-triple}
\newcommand{\TriplesA}{Py\-tha\-go\-re\-an $A$-triples}

\makeatletter
\def\opargproof[#1]{\par\noindent {\bf #1 }}
 \makeatother

\begin{document}
\begin{center}
{\LARGE\bf Congruent Number Elliptic Curves}\\[1.7ex] 
{\LARGE\bf with Rank at Least Two}

\medskip

{\small Lorenz Halbeisen}\\[1.2ex] 
{\scriptsize Department of Mathematics, ETH Zentrum,
R\"amistrasse\;101, 8092 Z\"urich, Switzerland\\ lorenz.halbeisen@math.ethz.ch}\\[1.8ex]
{\small Norbert Hungerb\"uhler}\\[1.2ex] 
{\scriptsize Department of Mathematics, ETH Zentrum,
R\"amistrasse\;101, 8092 Z\"urich, Switzerland\\ norbert.hungerbuehler@math.ethz.ch}
\end{center}

\hspace{5ex}{\small{\it key-words\/}: congruent number elliptic curves, integral Pythagorean triples}

\hspace{5ex}{\small{\it 2010 Mathematics Subject 
Classification\/}: {\bf 11G05}\,\ 11D09}

\begin{abstract}\noindent
We give an infinite family of congruent number elliptic curves, 
each with rank at least two, which are related to integral solutions
of $m^2=n^2+nl+l^2$.
\end{abstract}

\section{Introduction}
Elliptic curves and their geometric and algebraic structure have been a flourishing
field of research in the past. They find prominent applications in cryptography
and played a key role in the proof of Fermat's Last Theorem. A salient feature of the algebraic
structure of an elliptic curve is its rank. Among general elliptic curves, 
congruent number curves of high rank are of particular interest (see, \eg,~\cite{HighRank}). 
More difficult than finding an individual congruent number curve
of high rank is to find infinite families of such curves. Johnstone and Spearman~\cite{JS}
constructed such a family with rank at least three which is related
to rational points on the biquadratic curve $w^2 = t^4 + 14t^2 + 4$. In the present paper,
we show an elementary construction for an infinite family of 
congruent number curves of rank at least two which are related
to the quadratic diophantine equation $m^2=n^2+nl+l^2$, and which
have three integral points with positive $y$-coordinate 
on a straight line. Incidentally, 
some members of the family exhibit surprisingly high individual rank.
We start by fixing the notions and notations used througout the text.

A positive integer $A$ is called a {\bf congruent number} if 
$A$ is the area of a right-angled triangle with three rational 
sides. So, $A$ is congruent if and only if there exists a
rational Pythagorean tripel $(a,b,c)$ (\ie, $a,b,c\in\Q$,
$a^2+b^2=c^2$, and $ab\neq 0$), such that $\frac{ab}2=A$.
The sequence of integer congruent numbers starts with
$$
5, 6, 7, 13, 14, 15, 20, 21, 22, 23, 24, 28, 29, 30, 31, 34, 37,\ldots
$$
(see, \eg, the On-Line Encyclopedia of Integer Sequences~\cite{oeisA003273}). 
For example, $A=7$ is a congruent number, 
witnessed by the rational Pythagorean triple $$\Bigl(\frac{24}{5}\,,
\frac{35}{12}\,,\frac{337}{60}\Bigr).$$

It is well-known that $A$ is a congruent number if
and only if the cubic curve $$C_A:\ y^2=x^3-A^2 x$$
has a rational point $(x_0,y_0)$ with $y_0\neq 0$.
The cubic curve $C_A$ is called a {\bf congruent number elliptic curve}
or just {\bf congruent number curve}.
This correspondence between rational points on congruent number curves and 
rational Pythagorean triples can be made explicit as follows:
Let 
$$
C(\Q):= \{(x,y,A)\in \Q\times\Q^*\times \Z^*:y^2=x^3-A^2x\},
$$
where $\Q^*:=\Q\setminus\{0\}, \Z^*:=\Z\setminus\{0\}$, and
$$
P(\Q):=\{(a,b,c,A)\in \Q^3\times\Z^*:a^2+b^2=c^2\ \textsl{and\/}\ ab=2A\}.
$$
Then, it is easy to check that
\begin{equation}\label{psi}
\begin{aligned}
\psi\ :\ \quad P(\Q)&\ \to\  C(\Q)\\
(a,b,c,A)&\ \mapsto \ \Bigl(\frac{b(b+c)}{2}\,,\,
\frac{b^2(b+c)}{2}\,,\,A\Bigr)
\end{aligned}
\end{equation}
is bijective  and
\begin{equation}\label{psi-1}
\begin{aligned}
\psi^{-1}\ :\qquad C(\Q)&\ \to\  P(\Q)\\
        (x,y,A)&\ \mapsto\  \Bigl(\frac{2x A}{y}\,,\;
        \frac{x^2-A^2}{y}\,,\;\frac{x^2+A^2}{y}\,,\,A\Bigr).
\end{aligned}
\end{equation}

For positive integers $A$, a triple $(a,b,c)$ of 
non-zero rational numbers is called a {\tripleAbf} 
if $a^2+b^2=c^2$ and $A=\big{|}\frac{ab}{2}\big{|}$.
Notice that if $(a,b,c)$ is a {\tripleA}, then $A$ 
is a congruent number and $|a|,|b|,|c|$ are the 
lengths of the sides of a right-angled triangle 
with area $A$. Notice also that we allow $a,b,c$
to be negative.

If $a,b,c$ are positive integers such that $a^2+b^2=c^2$ and $A=\frac{ab}{2}$ 
is integral, then the triple $(a,b,c)$ is a called a {\TripleAbf}. 
For any positive integers $m$ and $n$ with $m>n$, the triple
$$\bigl(\,\underset{\text{\small $a$}}{\underbrace{\,\mathstrut 2mn\,}}\,,\; 
\underset{\text{\small $b$}}{\underbrace{\mathstrut m^2-n^2}}\,,\; 
\underset{\text{\small $c$}}{\underbrace{\mathstrut m^2+n^2}}\,\bigr)$$ 
is a {\TripleA}. In this case, we obtain $A=mn(m^2-n^2)$ and
\begin{equation}\label{eq-psi}
\psi(a,b,c,A)=\bigl(
\underset{\text{\small $x$}}{\underbrace{m^2(m^2-n^2)}}\,,\;
\underset{\text{\small $y$}}{\underbrace{m^2(m^2-n^2)^2}}\,,\;A\bigr)\,.
\end{equation}
In particular, the point $(x,y)$ on $C_A$ which corresponds to the 
{\TripleA} $(a,b,c)$ is an integral point. 

Concerning the equation
$$m^2=n^2+nl+l^2\,,$$ we would like to mention the following fact 
(see Dickson~\cite[Exercises\,XXII.2, p.\,80]{Dickson} or Cox~\cite[Chapter\,1]{Cox}): 
\begin{fct}\label{fct:m2}
 Let $p_1<p_2<\ldots <p_j$ be primes, such that
  $p_i\equiv 1\mod 6$ for $1\le i\le j$, and let $$m=\prod_{i=1}^j p_i.$$
 
 Then the number of positive, integral solutions $l<n$ of
 $$m=n^2+nl+l^2$$ is $2^{j-1}$. By definition of $m$, for each 
 integral solution of $m=n^2+nl+l^2$, $n$ and $l$ are relatively prime,
 denoted $(n,l)=1$.
 
 Moreover, the number of positive, integral solutions $l<n$ of
 $$m^2=n^2+nl+l^2$$ is $\frac{3^j-1}2.$
 Among the $\frac{3^j-1}2$ integral solutions $l<n$ of
 $m^2=n^2+nl+l^2$ we find $2^{j+1}$ solutions with $(n,l)=1$. 
 In particular, if $j=1$ and $p\equiv 1\mod 6$, then
 the solution in positive integers $n<l$ of 
 $$p^2=n^2+nl+l^2$$ is unique and $(n,l)=1$.
\end{fct}

For a geometric representation of integral solutions of $x^2+xy+y^2=m^2$,
see Halbeisen and Hungerb\"uhler~\cite{HHAnning}.

If $m,n,l$ are positive integers such that $m^2=n^2+nl+l^2$, then, for $k:=n+l$,
each of the following three triples
$$\bigl(\,\underset{\text{\small $a_1$}}{\underbrace{\,\mathstrut 2mn\,}}\,,\; 
\underset{\text{\small $b_1$}}{\underbrace{\mathstrut m^2-n^2}}\,,\; 
\underset{\text{\small $c_1$}}{\underbrace{\mathstrut m^2+n^2}}\,\bigr)\,,$$
$$\bigl(\,\underset{\text{\small $a_2$}}{\underbrace{\,\mathstrut 2ml\,}}\,,\; 
\underset{\text{\small $b_2$}}{\underbrace{\mathstrut m^2-l^2}}\,,\; 
\underset{\text{\small $c_2$}}{\underbrace{\mathstrut m^2+l^2}}\,\bigr)\,,$$
$$\bigl(\,\underset{\text{\small $a_3$}}{\underbrace{\,\mathstrut 2mk\,}}\,,\; 
\underset{\text{\small $b_3$}}{\underbrace{\mathstrut k^2-m^2}}\,,\; 
\underset{\text{\small $c_3$}}{\underbrace{\mathstrut k^2+m^2}}\,\bigr)\,,$$ 
is a {\TripleA} for $A=mn(m^2-n^2)=ml(m^2-l^2)=km(k^2-m^2)$ (see Hungerb\"uhler~\cite{Noebi}).
In particular, with $m,n,l$ and~(\ref{eq-psi}) we obtain three distinct integral points on $C_A$.

Let us now turn back to the curve $C_A$.
It is convenient to consider the curve $C_A$ in the
projective plane $\R P^2$, where the curve is given by
$$
C_A :\ y^2z = x^3-A^2xz^2.
$$
On the points of $C_A$, one can define a commutative, binary, 
associative operation ``$+$'', where $\nO$, the neutral 
element of the operation, is the projective point $(0,1,0)$
at infinity. More formally, if $P$ and $Q$ are two points on $C_A$, 
then let $P\dritt Q$ be the third intersection point of
the line through $P$ and $Q$ with the curve $C_A$. 
If $P=Q$, the line through $P$ and $Q$ is replaced with the tangent in $P$.
Then
$P+Q$ is defined by stipulating 
$$P+Q\;:=\;\nO\dritt (P\dritt Q),$$
where for a point $R$ on $C_A$, $\nO\dritt R$ is the point reflected across the $x$-axis.
The following figure shows the congruent number curve $C_A$ for
$A=5$, together with two points $P$ and $Q$ and their sum $P+Q$.
%

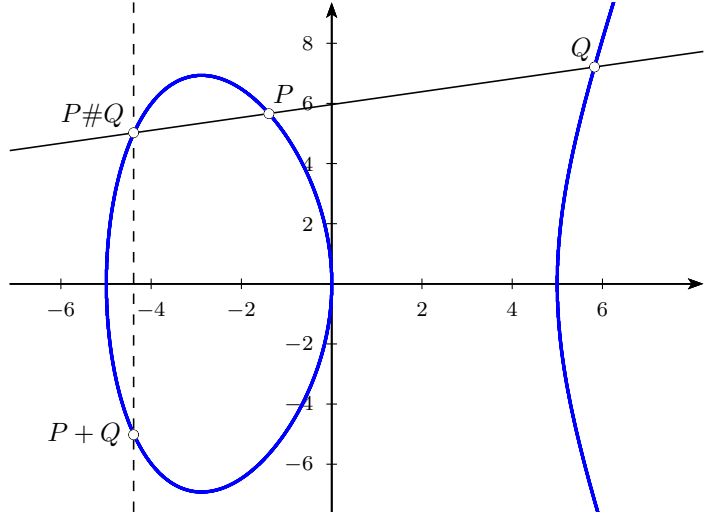
\begin{center}
\psset{xunit=.6cm,yunit=.4cm,algebraic=true,dimen=middle,dotstyle=o,dotsize=5pt 0,linewidth=1.6pt,arrowsize=3pt 2,arrowinset=0.25}
\begin{pspicture*}(-7.1329967371431815,-7.588280903625012)(8.234736979744335,9.360344080789211)
\psaxes[labelFontSize=\scriptstyle,xAxis=true,yAxis=true,Dx=2.,Dy=2.,ticksize=-2pt 2pt,subticks=1,linewidth=.6pt,]{->}(0,0)(-7.1329967371431815,-7.588280903625012)(8.234736979744335,9.360344080789211)
\psplotImp[linewidth=1.2pt,linecolor=blue,stepFactor=0.1](-9.0,-9.0)(9.0,10.0){1.0*y^2+25.0*x^1-1.0*x^3}
\psplot[linewidth=.6pt]{-7.1329967371431815}{8.234736979744335}{(-43.01566031988557-1.5476943995438228*x)/-7.215284285929719}
\psline[linewidth=.6pt,linestyle=dashed,dash=4pt 4pt](-4.388182328551535,-7.588280903625012)(-4.388182328551535,9.360344080789211)
\begin{small}
\psdots[dotsize=4pt 0](5.824738905621608,7.211160924647724)
\rput[bl](5.3,7.4){$Q$}
\psdots[dotsize=4pt 0](-1.3905453803081107,5.663466525103901)
\rput[bl](-1.3,6){$P$}
\psdots[dotsize=4pt 0](-4.388182328551535,5.020466785549749)
\rput[bl](-6.,5.2){$P\dritt Q$}
\psdots[dotsize=4pt 0](-4.388182328551535,-5.020466785549749)
\rput[bl](-6.3,-5.4){$P+Q$}
\end{small}
\end{pspicture*}
\end{center}
%
More explicitly, for two points $P=(x_0,y_0)$ and $Q=(x_1,y_1)$ on
a congruent number curve $C_A$, the point $P+Q=(x_2,y_2)$ is given by
the following formulas: 
\begin{itemize}
\item If $x_0\neq x_1$, then 
$$x_2=\lambda^2-x_0-x_1,\qquad y_2=\lambda(x_0-x_2)-y_0,$$
where $$\lambda:=\frac{y_1-y_0}{x_1-x_0}.$$
\item If $P=Q$, \ie, $x_0=x_1$ and $y_0=y_1$, then 
\begin{equation*}\label{eq:2P}
x_2=\lambda^2-2x_0,\qquad y_2=3x_0\lambda-\lambda^3-y_0,
\end{equation*}
where 
\begin{equation*}\label{eq:lambda}
\lambda:=\frac{3x_0^2-A^2}{2y_0}.
\end{equation*}
Below we shall write $2*P$ instead of $P+P$. 

\item If $x_0=x_1$ and $y_0=-y_1$, then $P+Q:=\nO$. In particular,
$(0,0)+(0,0)=(A,0)+(A,0)=(-A,0)+(-A,0)=\nO$. 
\item Finally,
we define $\nO+P:=P$ and $P+\nO:=P$ for any point $P$, in particular, 
$\nO+\nO=\nO$.
\end{itemize}
With the operation~``$+$'',
$(C_A,+)$ is an abelian group with neutral element $\nO$.
Let $C_A(\Q)$ be the set of rational points on $C_A$ together
with $\nO$. It is easy to see that $\bigl(C_A(\Q),+\bigr)$. 
is a subgroup of $(C_A,+)$. Moreover, it is well known that 
the group $\bigl(C_A(\Q),+\bigr)$ is finitely generated.
One can readily check that the three points $(0,0)$ and
$(\pm A,0)$ are the only points on $C_A$ of order~$2$,
and one easily finds other points of finite order on $C_A$.
However, it is well known that if
$A$ is a congruent number and $(x_0,y_0)$ is a rational
point on $C_A$ with $y_0\neq 0$, then the order of $(x_0,y_0)$
is infinite. In particular, if there exists one {\tripleA}, then
there exist infinitely many such triples (for an elementary proof
of this result, which is based on a theorem of Fermat's, see
Halbeisen and Hungerb\"uhler~\cite{Fermat}).
Furthermore, {\sc Mordell's Theorem}
states that the group of rational points on $C_A$ is
finitely generated, and by the
{\sc Fundamental Theorem of Finitely Generated Abelian Groups}, 
the group of rational points on an elliptic curve is isomorphic
to some group of the form
$$\underset{\text{\scriptsize torsion group}}
{\underbrace{\Z/n_1\Z\times\ldots\times\Z/n_k\Z}}\times\Z^r,$$
where $n_1,\ldots,n_k$ and $r$ are positive integers. The group 
$\Z/\Z_{n_1}\times\ldots\times\Z/\Z_{n_k}$, which is generated by rational
points of finite order, is the so-called
\emph{torsion group}, and $r$ is 
called the \emph{rank\/} of the curve. Now, since $C_A$ does not
have rational points of finite order besides the points $(0,0)$ and
$(\pm A,0)$, the torsion group of $C_A$ is isomorphic to 
$\Z/2\Z\times\Z/2\Z$. 

Based on integral solutions of $m^2=n^2+nl+l^2$, we will show that there
are infinitely many congruent number curves $C_A$ with rank at least two
(for congruent number curves with rank at least three see
Johnstone and Spearman~\cite{JS}). 

\section{Rank at Least Two}

In order to ``compute'' the rank of a curve of the form
$$\Gamma: y^2=x^3+Bx,$$ according to 
Silverman and Tate\;\cite[Chapter\,III.6.]{SilvermanTate},
we first have to write down several equations of the form
\begin{eqnarray}
 b_1 M^4+b_2e^2&=&N^4\label{eqn:A}\\[1.2ex]
 \bar{b}_1\bar{M}^4+\bar{b}_2\bar{e}^2&=&\bar{N}^4, \label{eqn:barA}
\end{eqnarray}
namely one for each factorisation $B=b_1b_2$ and $-4B=\bar{b}_1\bar{b}_2$, respectively,
where $b_1$ and $\bar{b}_1$ are square-free. Then we have to decide,
how many of these equations have integral solutions:
Let $\#\alpha(\Gamma)$ be 
the number of equations of the form~(\ref{eqn:A}) for which we find 
integral solutions $M,e,N$ with $e\neq 0$, and
let $\#\alpha(\bar\Gamma)$ be 
be the corresponding number with respect to equations of the form~(\ref{eqn:barA}).
Then, if $r>0$, $$2^r=\frac{\#\alpha(\Gamma)\cdot \#\alpha(\bar\Gamma)}{4}.$$
Moreover, one can show that if $(x,y)$ is a rational point on $\Gamma$, where $y\neq 0$,
then one can write that point in the form
\begin{equation}\label{eq-x_b1}
x=\frac{b_1 M^2}{e^2},\qquad y=\frac{b_1 M N}{e^3},
\end{equation}
where $M,e,N$ is an integral solution of an equation of the form~(\ref{eqn:A}),
and vice versa. The analogous statement holds for rational points on the curve 
$\bar\Gamma: y^2=x^3-4Bx$ with respect to equations of the form~(\ref{eqn:barA}).

Now we are ready to prove

\begin{thm}\label{thm:main}
Let $m,n,l$ be pairwise relatively prime positive integers, where
$m=\prod_{i=1}^j p_i$ is a product of pairwise distinct primes
$p_i\equiv 1\mod 6$ and $m^2=n^2+nl+l^2$. Furthermore, let $k:=n+l$
and let $A:=klmn$. Then the rank of 
the curve $$C_A: y^2=x^3-A^2x$$ is at least two.
\end{thm}

\begin{proof} Since we have at least one rational point $(x,y)$ on $C_A$
with $y\neq 0$, we know that the rank $r$ of $C_A$ is positive. So,
to show that the rank of the curve $C_A$ is at least two, it
is enough to show that $\#\alpha(C_A)\ge 9$. For this, we have
to show that there are integral solutions for~(\ref{eqn:A}) for at least~9 
distinct square-free integers $b_1$ dividing $-A^2$, or equivalently, 
we have to find at least~9 rational points on $C_A$, such that the~9 
corresponding integers $b_1$ are pairwise distinct, which we will do now.

Notice that because of~(\ref{eq-x_b1}), to compute $b_1$ from a rational point $P=(x,y)$ on $C_A$
with $x\neq 0$, it is enough to know the $x$-coordinate of $P$ and then
compute $x$~mod~${\Q^*}^2$ (\ie, we compute $x$ modulo squares of rationals).
The $x$-coordinates of the three integral points we get by~(\ref{psi}) from the three
{\TriplesA} $(a_1,b_1,c_1)$, $(a_2,b_2,c_2)$, $(a_3,b_3,c_3)$ generated by  
$m,n,l,k$, are
$$x_1=m^2(m^2-n^2)=m^2kl\,,\quad x_2=m^2(m^2-l^2)=m^2kn\,,\quad k^2(k^2-m^2)=k^2nl\,,$$ 
and modulo squares, this gives us three values for $b_1$,
namely $$b_{1,1}=kl\,,\quad b_{1,2}=kn\,,\quad b_{1,3}=nl\,.$$ 
Now, exchanging in each of the three {\TriplesA} the two catheti $a_i$ and $b_i$ (for $i=1,2,3$),
we obtain again three distinct integral points on $C_A$, whose $x$-coordinates modulo squares
give us $$b_{1,4}=mn\,,\quad b_{1,5}=ml\,,\quad b_{1,6}=mk\,.$$
Finally, if we replace each hypothenuse $c_j$ of these six {\TriplesA} with $-c_j$, we 
obtain again six distinct integral points on $C_A$, whose $x$-coordinates modulo squares
give us 
$$
\begin{array}{rclrclrcl}
  b_{1,7}&\hspace{-1.5ex}=\hspace{-1.5ex}&-kl\,,
  & b_{1,8}&\hspace{-1.5ex}=\hspace{-1.5ex}&-kn\,,
  & b_{1,9}&\hspace{-1.5ex}=\hspace{-1.5ex}&-nl\,\\[1.6ex]
  b_{1,10}&\hspace{-1.5ex}=\hspace{-1.5ex}&-mn\,,
  & b_{1,11}&\hspace{-1.5ex}=\hspace{-1.5ex}&-ml\,,& 
  b_{1,12}&\hspace{-1.5ex}=\hspace{-1.5ex}&-mk\,.
\end{array}
$$
In addition to these~$12$ integral points on $C_A$ (and the 
corresponding $b_1$'s), we have the two integral points $(\pm A,0)$ on $C_A$,
which give us $$b_{1,13}=klmn\qquad\text{and}\qquad b_{1,14}=-klmn.$$

Recall that, by assumption, $m$ is square-free and $k,l,n$ are pairwise relatively prime.
Therefore, if for some $i,j$ with $1\le i<j\le 14$, $b_{1,i}\equiv b_{1,j}$ modulo squares,
at least two of the integers $k,l,n$ are squares, say $n=u^2$, and $l=v^2$ or $k=v^2$. 
Then $$m^2=u^4+u^2v^2+v^4\quad\text{(in the case when $l=v^2$),}$$
or $$m^2=u^4-u^2v^2+v^4\quad\text{(in the case when $k=v^2$).}$$
If $l=v^2$, this implies that $u^2=1$ and $v=0$, or $u=0$ and $v^2=1$, and
if $k=v^2$, this implies that $u^2=1$ and $v=0$, $u=0$ and $v^2=1$, or $u^2=v^2=1$
(see, for example, Mordell~\cite[p.\,19{\sl f\/}]{MordellBook} or 
Euler~\cite[p.\,16]{Euler1780}).

So, at most one of the integers $k,l,n$ is a square, which implies that
$\#\alpha(C_A)\ge 14$ and this completes the proof.
\end{proof}

As an immediate consequence we get the following

\begin{cor} Let $m,n,l$ be as in {\THM}\;\ref{thm:main} and let
$q$ be a non-zero integer. Then the rank of 
the curve $C_{Aq^4}$ is at least two.
\end{cor}

\begin{proof}
Notice that if $m,n,l$ are such that $m^2=n^2+nl+l^2$, then, 
for $mq,nq,lq$, we have $(mq)^2=(nq)^2+nq\cdot lq+(lq)^2$, 
which implies that for $\tilde A=kq\cdot lq\cdot mq\cdot nq=Aq^4$,
the rank of the curve $C_{\tilde A}$ is at least two.
\end{proof}

\section{Odds and Ends}

As a matter of fact we would like to mention that the 
three integral points on $C_A$ which correspond to an integral solution 
of $m^2=n^2+nl+l^2$ lie on a straight line.

\begin{fct}
Let $m,n,l$ be positive integers such that $m^2=n^2+nl+l^2$, let
$k=n+l$, and let $A=klmn$. Then the three integral points
\begin{eqnarray*}
\bigl(\,\underset{x_1}{\underbrace{m^2(m^2-n^2)}}\,,\;
\underset{y_1}{\underbrace{m^2(m^2-n^2)^2}}\,\bigr)\,,\\[1.4ex]
\bigl(\,\underset{x_2}{\underbrace{m^2(m^2-l^2)}}\,,\;
\underset{y_2}{\underbrace{m^2(m^2-l^2)^2}}\,\bigr)\,,\\[1.4ex]
\bigl(\,\underset{x_3}{\underbrace{k^2(k^2-m^2)}}\,,\;
\underset{y_3}{\underbrace{k^2(k^2-m^2)^2}}\,\bigr)\,,
\end{eqnarray*}
on the curve $C_A$ lie on a straight line.
\end{fct}

\begin{proof}
For $i=2,3$ let $$\lambda_{1,i}:=\frac{y_i-y_1}{x_i-x_1}\,.$$ 
It is enough to show that $\lambda_{1,2}=\lambda_{1,3}$, or equivalently,
that $\lambda_{1,3}-\lambda_{1,2}=0$. Now, an easy calculation shows that
$\lambda_{1,2}=k^2$ and that $\lambda_{1,3}-k^2=\frac{0}{k(k-n)^3}=0$.
\end{proof}

As a last remark we would like to mention that with the help
of {\scriptsize \sagelogo} we found that some of the curves
which correspond to an integral solution of $m^2=n^2+nl+l^2$ have
rank~3 or higher. In fact, we found several curves of rank~$3$ or~$4$,
as well as the following curves of rank~$5$:

\begin{center}
\begin{tabular}{rcrrr}
$A=klmn$\hspace{1ex} &  \hspace{3ex}$m=\prod p_{i\mathstrut}$\hspace{3ex} & 
\hspace{5ex}$l$\hspace{1ex} & \hspace{9ex}$n$\hspace{1ex} & \hspace{3ex}$k=n+l$\\
\hline\\
%
$237\,195\,512\,400$ & 
$7\cdot 127$ & $464$ & $561$ & $1\,025$\\
$8\,813\,542\,297\,560$ & 
$7\cdot 13\cdot 37$ & $232$ & $3\,245$ & $3\,477$\\
%
$10\,280\,171\,942\,040$ & 
$37\cdot 67$ & ${741}$ & $2\,024$ & $2\,765$\\
%
$81\,096\,660\,783\,600$ & 
$37\cdot 103$ & $2\,139$ & $2\,261$ & $4\,400$\\
%
$225\,722\,120\,463\,840$ & 
$13\cdot 19\cdot 31$ & $505$ & $7\,392$ & $7\,897$\\
%
$457\,485\,316\,904\,280$ & 
$7\cdot 31\cdot 37$ & $895$ & $7\,544$ & $8\,439$\\
%
$5\,117\,352\,889\,729\,080$ & 
$67\cdot 223$ & $1\,551$ & $14\,105$ & $15\,656$\\
%
$281\,692\,457\,452\,791\,000$ & 
$79\cdot 409$ & $9\,064$ & $26\,811$ & $35\,875$\\
%
$24\,666\,188\,870\,481\,576\,600$ & 
$13\cdot 31\cdot 223$ & $46\,169$ & $57\,400$ & $103\,569$
\end{tabular}
\end{center}

It is possible that these curves might be candidates
for high rank congruent number elliptic curves (for another approach see 
Dujella, Janfada, Salami~\cite{HighRank}).
%

\providecommand{\bysame}{\leavevmode\hbox to3em{\hrulefill}\thinspace}
\providecommand{\MR}{\relax\ifhmode\unskip\space\fi MR }
\providecommand{\MRhref}[2]{%
  \href{http://www.ams.org/mathscinet-getitem?mr=#1}{#2}
}
\providecommand{\href}[2]{#2}

\end{document}